\documentclass[12 pt]{article}
\usepackage[pdftex]{graphicx}
\usepackage{amsmath}
\usepackage{mathtools}
\usepackage{amsthm}
\usepackage{amsfonts}
\usepackage{amssymb}
\usepackage[margin=1.0in]{geometry}
\usepackage{tikz-cd}
\usetikzlibrary{cd}
\usepackage{enumitem}

\makeatletter
\newtheorem*{rep@theorem}{\rep@title}
\newcommand{\newreptheorem}[2]{%
	\newenvironment{rep#1}[1]{%
		\def\rep@title{#2 \ref{##1}}%
		\begin{rep@theorem}}%
		{\end{rep@theorem}}}
\makeatother

\newtheorem{theorem}{Theorem}[section]
\newreptheorem{theorem}{Theorem}
\newtheorem{cor}[theorem]{Corollary}
\newtheorem{lemma}[theorem]{Lemma}
\newtheorem{prop}[theorem]{Proposition}

\theoremstyle{definition}
\newtheorem{defin}[theorem]{Definition}
\newtheorem{fact}[theorem]{Fact}

\newtheorem{que}[theorem]{Question}

\theoremstyle{remark}
\newtheorem*{rem}{Remark}

\newcommand{\flim}[1]{\mathrm{Flim}(#1)}

\newcommand{\fin}[1]{\mathrm{Fin}(#1)}
\newcommand{\fr}{Fra\"iss\'e }
\renewcommand{\phi}{\varphi}

\newcommand{\emb}[1]{\mathrm{Emb}(#1)}
\newcommand{\aut}[1]{\mathrm{Aut}(#1)}

\renewcommand{\b}{\backslash}
\newcommand{\Sa}{\mathrm{Sa}}

\begin{document}
	\title{Maximally highly proximal flows}
	\author{Andy Zucker}
	\date{February 2019; revised October 2019}
	\maketitle
	
\begin{abstract}
	For $G$ a Polish group, we consider $G$-flows which either contain a comeager orbit or have all orbits meager. We single out a class of flows, the \emph{maximally highly proximal} (MHP) flows, for which this analysis is particularly nice. In the former case, we provide a complete structure theorem for flows containing comeager orbits, generalizing theorems of Melleray-Nguyen Van Th\'e-Tsankov and Ben Yaacov-Melleray-Tsankov. In the latter, we show that any minimal MHP flow with all orbits meager has a metrizable factor with all orbits meager, thus ``reflecting'' complicated dynamical behavior to metrizable flows. We then apply this to obtain a structure theorem for Polish groups whose universal minimal flow is distal. 
	\let\thefootnote\relax
	\let\thefootnote\relax\footnote{The author was supported by NSF Grant no.\ DMS 1803489.}
\end{abstract}
	
\section{Introduction}

Let $G$ be a Polish group. A \emph{$G$-flow} is a compact Hausdorff space equipped with a continuous (right) $G$-action $X\times G\to X$. If $X$ and $Y$ are $G$-flows, a map $\phi\colon X\to Y$ is a \emph{$G$-map} if $\phi$ is continuous and respects the $G$-actions. A \emph{subflow} of a $G$-flow $X$ is any non-empty closed invariant subspace $Y\subseteq X$. We say $X$ is \emph{minimal} if the only subflow of $X$ is $X$ itself. Equivalently, $X$ is minimal if for every $x\in X$, the orbit $x\cdot G\subseteq X$ is dense. Notice that if $\phi\colon X\to Y$ is a $G$-map, then the image $\phi[X]\subseteq Y$ is a subflow; if $X$ is minimal, so is $\phi[X]$, and if $Y$ is minimal, then $\phi$ is surjective. We often call a surjective $G$-map a \emph{factor}.

By a classical theorem of Ellis, there is a \emph{universal minimal flow} $M(G)$; this is a minimal $G$-flow which admits a $G$-map onto any other minimal $G$-flow, and $M(G)$ is unique up to isomorphism. The study of $M(G)$ is useful because it captures information about all minimal $G$-flows. For instance, if $M(G)$ is metrizable, then every minimal $G$-flow is metrizable, and if $M(G)$ has a (necessarily unique) comeager orbit, then so does every minimal $G$-flow \cite{AKL}. However, $M(G)$ is often very complicated; for example, if $G$ is locally compact, then $M(G)$ is never metrizable, and all of its orbits are meager. However, there are Polish groups $G$ for which $M(G)$ is a singleton, and many others for which $M(G)$ is metrizable and has a concrete description. See \cite{KPT} for several examples of these phenomena. 

The starting point of this paper is the following theorem, first proved by the author \cite{ZucMetr} in the case that $G$ is non-Archimedean, and then by Ben Yaacov, Melleray, and Tsankov \cite{BYMT} for general Polish groups.
\vspace{2 mm}

\begin{fact}
	\label{Fact:MetrizableImpliesComeager}
	If $G$ is a Polish group and $M(G)$ is metrizable, then $M(G)$ has a comeager orbit.
\end{fact} 
\vspace{2 mm}

This theorem along with the structure theorem due to Melleray, Nguyen Van Th\'e, and Tsankov \cite{MNT} provide a complete understanding of the structure of $M(G)$ when it is metrizable. However, the property of $M(G)$ having a comeager orbit remained less well understood. Indeed, it was only recently shown, by an example of Kwiatkowska \cite{Kw}, that the converse of Fact~\ref{Fact:MetrizableImpliesComeager} does not hold. 

The study of $M(G)$ is often undertaken by attempting to understand the \emph{Samuel compactification} $\Sa(G)$, the Gelfand space of the bounded left uniformly continuous functions on $G$. The group $G$ canonically embeds into $\Sa(G)$, and for any $G$-flow $X$ and any $x\in X$, there is a unique $G$-map $\lambda_x\colon \Sa(G)\to X$ with $\lambda_x(1_G) = x$. In particular, any minimal subflow of $\Sa(G)$ is isomorphic to $M(G)$. The main technical tool introduced in \cite{BYMT} is to view $\Sa(G)$ as a \emph{topometric space}, a topological space endowed with a possibly finer metric $\partial$ which interacts with the topology in nice ways. Letting $\partial$ denote this finer metric, the authors of \cite{BYMT} show that if $M\subseteq \Sa(G)$ is a compact metrizable subspace, then $\partial|_M$ is a compatible metric. When the metrizable $M\subseteq \Sa(G)$ is a minimal subflow, the properties of the metric $\partial|_M$ allow them to show that $M$ has a comeager orbit. However, much remained unclear about this metric, especially when $M(G)$ is non-metrizable. Namely, if $M\subseteq \Sa(G)$ is a minimal subflow, can we define $\partial|_M$ just using the dynamics of $M$? 

This paper singles out a class of flows, the \emph{maximally highly proximal flows}, or \emph{MHP flows}, which all admit a canonical topometric structure. In particular, $M(G)$ and $\Sa(G)$ are both MHP, and the topometric on $M(G)$ agrees with the metric inherited by any minimal subflow of $\Sa(G)$. Using this topometric structure, we provide a structure theorem for MHP flows with a comeager orbit. Here, a \emph{compatibility point} is a point in $X$ where the topology and the metric coincide (see Definition~\ref{Def:CompatiblePoint}). 
\vspace{2 mm}

\begin{reptheorem}{IntroThm:MHP}
	Let $X$ be an MHP flow. The following are equivalent.
	\vspace{-2 mm}
	
	\begin{enumerate}
		\item 
		$X$ has a compatibility point with dense orbit.
		\item 
		The set $Y\subseteq X$ of compatibility points is comeager, Polish, and contains a point with dense orbit.
		\item 
		$X$ has a comeager orbit.
		\item 
		$X\cong \Sa(H\backslash G)$ for some closed subgroup $H\subseteq G$ (see Section~\ref{Sec:Samuel})
	\end{enumerate}
\end{reptheorem}
\vspace{2 mm}

In Theorem~\ref{Thm:ExAmSub}, we generalize the main result of \cite{MNT} by considering the case that $X \in \{M(G), \Pi(G), \Pi_s(G)\}$, where $\Pi(G)$ and $\Pi_s(G)$ are the \emph{universal minimal proximal flow} and the \emph{Furstenberg boundary}, respectively. In the first and third case, we show that the closed subgroup $H$ appearing in item $(4)$ is extremely amenable or amenable, respectively, and in the second case we present a partial result towards showing that $H$ is strongly amenable. 

As an application of Theorem~\ref{IntroThm:MHP}, we prove the following ``reflection'' theorem, which shows that complicated dynamical behavior of the group $G$ already appears in the realm of metrizable flows. Note that in minimal flows, all orbits are either meager or comeager.
\vspace{2 mm}

\begin{reptheorem}{Thm:Reflection}
	Let $X$ be a minimal MHP flow all of whose orbits are meager. Then there is a factor $\phi\colon X\to Y$ so that $Y$ is metrizable and also has all orbits meager.
\end{reptheorem}
\vspace{2 mm}

This theorem was first suggested in \cite{BYMT}, but in private communication with the authors, it was realized that the problem remained open. 

As an application of Theorem~\ref{Thm:Reflection}, we give a complete characterization of when $M(G)$ is distal in Theorem~\ref{Thm:DistalFlows} and Corollary~\ref{Cor:DistalForm}. The theorem says that if $M(G)$ is distal, then $M(G)$ is metrizable. Then using results from \cite{MNT}, the corollary shows that any such $G$ has a normal, extremely amenable subgroup $H$ with $M(G)\cong H\backslash G$.

\subsection*{Acknowledgements} I thank Todor Tsankov for many helpful discussions, including the suggestion that Theorem 7.5 is true. Some of the work here builds on work in my Ph.D.\ thesis, and I thank Clinton Conley for his guidance in its completion. I also thank the referee for many helpful suggestions on an earlier draft.

\subsection*{Notation}

We will use some non-standard notation. The phrases ``non-empty open subset of,'' ``open neighborhood of,'' etc.\ occur often enough that we introduce some notation for this. If $X$ is a topological space, then $A\subseteq_{op} X$ will mean that $A$ is a non-empty open subset of $X$. If $x\in X$, we write $x\in_{op} A$ or $A\ni_{op} x$ to mean that $A\subseteq X$ is an open neighborhood of $x$. Omitting the ``op'' subscript does not mean that a given set is not open; it is just an easy way to introduce and/or emphasize open sets.

Other notation is mostly standard. We write $\omega = \{0,1,2,\ldots\}$, and we identify a non-negative integer with the set of its predecessors, i.e.\ $n = \{0,\ldots, n-1\}$. If $f\colon X\to Y$ is a function and $K\subseteq X$, we set $f[K] := \{f(x): x\in K\}$. All topological spaces we consider are Hausdorff.
\vspace{2 mm}

\section{Topometric spaces}

This short section collects the background material on topometric spaces that we will need going forward. Most of the material here can be found in \cite{BY} or \cite{BYM}.
\vspace{2 mm}

\begin{defin}
	\label{Def:TopometricSpace}
	A \emph{compact topometric space} is a triple $(X, \tau, \partial)$, where $(X, \tau)$ is a compact Hausdorff space and $\partial$ is a metric which is \emph{lower semi-continuous}, meaning that for every $c\geq 0$, the set $\{(p, q)\in X^2: \partial(p, q)\leq c\}$ is $(\tau\times\tau)$-closed.
\end{defin}
\vspace{2 mm}

Note that the metric need not agree with the underlying topology. As a convention, when discussing a topometric space, topological vocabulary will refer to $\tau$, while metric vocabulary will refer to $\partial$.
\vspace{2 mm}

\begin{fact}
	\label{Fact:TopometricFacts}
	Let $(X, \tau, \partial)$ be a compact topometric space.
	\vspace{-2 mm}
	
	\begin{enumerate}
		\item 
		The metric $\partial$ is finer than the topology.
		\item 
		The metric $\partial$ is complete.
	\end{enumerate}
\end{fact}
\vspace{0 mm}

\begin{rem}
	\label{Rem:GenTopo}
	One can also define topometric spaces where the underlying topological space is not compact. One then includes item $(1)$ above in the definition.
\end{rem}
\vspace{2 mm}

The following fact will be needed going forward.
\vspace{2 mm}

\begin{fact}[Ben Yaacov \cite{BY}]
	\label{Fact:ExistLipschitz}
	Let $(X, \tau, \partial)$ be a compact topometric space. Then if $K, L\subseteq X$ are closed with $\partial(K, L) > r$, then there is a continuous, $1$-Lipschitz function $f\colon X\to [0,1]$ with $f[K] = \{0\}$ and $f[L] = \{r\}$. 
\end{fact}
\vspace{2 mm}

If $(X, \tau, \partial)$ is a compact topometric space, $K\subseteq X$, and $c> 0$, we define $K(c) := \{p\in X: \partial(p, K) < c\}$ and $K[c] := \{p\in X: \partial(p, K) \leq c\}$. If $K = \{p\}$ for some $p\in X$, we just write $p(c)$ or $p[c]$, respectively.
\vspace{2 mm}

\begin{defin}[\cite{BYM}, Def. 1.25]
	\label{Def:Adequate}
	A topometric space $(X, \tau, \partial)$ is called \emph{adequate} if for every open $A\subseteq X$ and every $c > 0$, we have $A(c)$ open.
\end{defin}
\vspace{2 mm}

We will prove (see Theorem~\ref{Thm:Adequate}) that the topometric spaces we consider in this paper are all adequate.
\vspace{2 mm}

\section{Maximally highly proximal flows}

Throughout this section, $G$ will denote a fixed Polish group. We let $d_G$ denote a compatible left-invariant metric of diameter $1$, and for $c> 0$, we set $U_c := \{g\in G: d_G(1_G, g) < c\}$. We will frequently and without explicit mention make use of the inclusion $U_cU_\epsilon \subseteq U_{c+\epsilon}$.
\vspace{2 mm}

\begin{defin}
	\label{Def:MHP}
	Let $X$ be a $G$-flow. We say that $X$ is \emph{maximally highly proximal}, or \emph{MHP}, if for every $A\subseteq_{op} X$, every $x\in \overline{A}$, and every $c > 0$, we have $x\in \mathrm{int}(\overline{AU_c})$.
\end{defin}	

\subsection{Highly proximal extensions}

The name MHP comes from the notion of a \emph{highly proximal} extension. If $\phi\colon Y\to X$ is a surjective $G$-map, we define the \emph{fiber image} of $B\subseteq_{op} Y$ to be $\phi_{fib}(B) := \{x\in X: \phi^{-1}(\{x\})\subseteq B\}$. The set $\phi_{fib}(B)$ is always open whenever $B\subseteq_{op} Y$, but possibly empty. We call $\phi$  \emph{highly proximal} if $\phi_{fib}(B)\neq \emptyset$ for every $B\subseteq_{op} Y$. The composition of highly proximal maps is also highly proximal. Also notice that if $X$ is minimal and $\phi\colon Y\to X$ is highly proximal, then $Y$ is also minimal. More precisely, if $X$ is any $G$-flow and $x\in X$ has dense orbit, then if $\phi\colon Y\to X$ is any highly proximal extension, then any $y\in \phi^{-1}(\{x\})$ also has dense orbit. 

To motivate why this notion receives the name ``highly proximal,'' it is helpful to compare this to the notion of a proximal extension. A $G$-map $\phi\colon Y\to X$ is called \emph{proximal} if for any $y_0, y_1\in Y$ with $\phi(y_0) = \phi(y_1)$, we can find a net $g_i$ from $G$ and $z\in Y$ with $\lim y_0g_i = \lim y_1g_i = z$. Now suppose that $X$ is minimal and that $\phi\colon Y\to X$ is highly proximal. Then $\phi$ is proximal. To see this, let $y_0, y_1\in Y$ with $\phi(y_0) = \phi(y_1) = x$. Fix any $z\in Y$, and let $\{B_i: i\in I\}$ be a base of neighborhoods of $z$. For each $B_i$, we have $\phi_{fib}(B_i):= A_i\neq \emptyset$. By minimality, let $g_i\in G$ be such that $xg_i\in A_i$. Then we see that $\lim yg_i = z$ for any $y\in \phi^{-1}(\{x\})$, so in particular for $y_0$ and $y_1$. In fact, this is historically the definition of a highly proximal extension.
\vspace{2 mm}

\begin{fact}[\cite{AG}, p.\ 733]
	\label{Fact:HPOriginal}
	Let $X$ be a minimal flow. Then the extension $\phi\colon Y\to X$ is highly proximal iff for any $x\in X$, there is a net $g_i\in G$ and a point $y\in Y$ with $\phi^{-1}(\{xg_i\})\to \{y\}$, i.e. for any $B\ni_{op} y$, we eventually have $\phi^{-1}(xg_i)\subseteq B_i$.
\end{fact} 
\vspace{2 mm}

In the case that $X$ is a minimal flow, Auslander and Glasner \cite{AG} prove the existence and uniqueness of a \emph{universal highly proximal extension}; this is a highly proximal $G$-map $\pi_X\colon S_G(X)\to X$ so that for any other highly proximal $\phi\colon Y\to X$, there is a $G$-map $\psi\colon S_G(X)\to Y$ with $\pi_X = \phi\circ \psi$.
\vspace{0 mm} 

\begin{center}
\begin{tikzcd}
	S_G(X)\arrow[d, "\pi_X"] \arrow[rd, dashrightarrow, "\psi"]\\
	X & Y\arrow[l, "\phi"]
\end{tikzcd}
\end{center}

Such a $\psi$ is necessarily also highly proximal.

The notion of a universal highly proximal extension was generalized to any $G$-flow in \cite{ZucProx}, where an explicit construction is given. We briefly review this construction here, referring to \cite{ZucProx} for all proofs. 
\vspace{2 mm}

\begin{defin}
	\label{Def:NearUlt}
	Fix a $G$-flow $X$, and write $\mathrm{op}(X) := \{A: A\subseteq_{op} X\}$. A collection $p\subseteq \mathrm{op}(X)$ is called a \emph{near ultrafilter} if: 
	\vspace{-2 mm}
	
	\begin{enumerate}
		\item 
		For every $k < \omega$, $A_0,...,A_{k-1}\in p$, and $c > 0$, we have $\bigcap_{i < k} A_iU_c \neq \emptyset$. We call this property the \emph{Near Finite Intersection Property}, or NFIP.
		\item 
		$p$ is maximal with respect to satisfying item (1).
	\end{enumerate}
\end{defin}
\vspace{2 mm}

Let $S_G(X)$ denote the collection of near ultrafilters on $\mathrm{op}(X)$. For $A\subseteq_{op} X$, we set $C_A = \{p\in S_G(X): A\in p\}$ and $N_A = \{p\in S_G(X): A\not\in p\}$. We endow $S_G(X)$ with a compact Hausdorff topology given by the base $\{N_A := A\subseteq_{op} X\}$. For $p\in S_G(X)$, a base of (not necessarily open) neighborhoods of $p$ is given by $\{C_{AU_\epsilon}: A\in p, \epsilon > 0\}$. The group $G$ acts on $S_G(X)$ in the obvious way, where $A\in pg$ iff $Ag^{-1}\in p$. We also have a canonical $G$-map $\pi_X\colon S_G(X)\to X$, where $\pi_X(p) = x$ iff for every $A\ni_{op} x$, we have $A\in p$.
\vspace{2 mm}

\begin{fact}
	\label{Fact:UnivHP}
	$\pi_X\colon S_G(X)\to X$ is the universal highly proximal extension of $X$.
\end{fact}
\vspace{2 mm}

In particular, the map $\pi_{S_G(X)}\colon S_G(S_G(X))\to S_G(X)$ is an isomorphism. The construction of the space of near ultrafilters in fact works on any \emph{$G$-space}, where the underlying space $X$ need not be compact. While in this generality we do not get the map $\pi_X$, we will still refer to the universal highly proximal extension of the $G$-space $X$, and the construction will still be idempotent. A remark that will be useful later is that if $Y\subseteq X$ is a dense $G$-invariant subspace of a $G$-space $X$, then $S_G(X)$ and $S_G(Y)$ coincide.
\vspace{2 mm}

\begin{prop}
	\label{prop:MHP}
	The $G$-flow $X$ is MHP iff the universal highly proximal extension \newline $\pi_X\colon S_G(X)\to X$ is an isomorphism.
\end{prop}

\begin{proof}
First let $X$ be any $G$-flow. Fix $p\in S_G(X)$, and set $x = \pi_X(p)$. Then we must have $p\subseteq \mathcal{F}_x := \{A\subseteq_{op} X: x\in \overline{A}\}$. To see why, if $x\not\in \overline{A}$, we can find $B\ni_{op} x$ and $c> 0$ with $AU_c\cap BU_c = \emptyset$. As $B\in p$ by definition of the map $\pi_X$, we cannot have $A\in p$.

Now suppose the $G$-flow $X$ is MHP. Then for every $x\in X$, we have that $\mathcal{F}_x$ has the NFIP, so is a near ultrafilter. It follows that if $p\in S_G(X)$ with $\pi_X(p) = x$, then we in fact have $p = \mathcal{F}_x$. In particular, the map $\pi_X$ is injective, hence an isomorphism.

Conversely, suppose $X$ is not MHP. Find some $x\in X$, $B\subseteq_{op} X$ with $x\in \overline{B}$, and $c> 0$ with $x\not\in \mathrm{int}(\overline{BU_c})$. Setting $C = X\setminus \overline{BU_c}$, we have $x\in \overline{C}$. Notice that $BU_{c/2}\cap CU_{c/2} = \emptyset$, so $B$ and $C$ can never belong to the same near ultrafilter. Set $\mathcal{G}_x := \{A\subseteq_{op} X: x\in A\}$. Let $p\in S_G(X)$ extend $\mathcal{G}_x\cup \{B\}$, and let $q\in S_G(X)$ extend $\mathcal{G}_x\cup\{C\}$. Then $p\neq q$ and $\pi_X(p) = \pi_X(q) = x$.
\end{proof}
\vspace{2 mm}

\subsection{Examples of MHP flows}
	
We now collect some examples of MHP flows. Of course, the universal highly proximal extension of any $G$-space is an MHP flow, but it will be useful to have some explicit examples in mind.

\subsubsection{Samuel compactifications}
\label{Sec:Samuel}

Let $H\subseteq G$ be a closed subgroup, and let $H\backslash G$ denote the right coset space. We equip $H\backslash G$ with the metric that it inherits from $G$, which we also denote by $d_G$. Explicitly, if $Hg\in H\backslash G$, the ball of radius $\epsilon > 0$ around $Hg$ is given by $HgU_\epsilon$. Then the \emph{Samuel compactification} $\Sa(H\backslash G)$ is the Gelfand space of the bounded uniformly continuous functions on $H\b G$. It is a $G$-flow characterized by the property that for any $G$-flow $Y$ containing a point $y\in Y$ with $y\cdot h = y_0$ for every $h\in H$, then there is a (necessarily unique) $G$-map $\phi\colon \Sa(H\backslash G)\to Y$ with $\phi(H) = y$. In the case $H = \{1_G\}$, we often write $yp := \phi(p)$. We identify $H\backslash G$ with its image under the canonical embedding $i\colon H\backslash G\hookrightarrow \Sa(H\backslash G)$.
		
To see that $\Sa(H\backslash G)$ is MHP, suppose $\psi\colon X\to \Sa(H\backslash G)$ were highly proximal. Using the universal property of $\Sa(H\backslash G)$, it is enough to show that $\psi^{-1}(\{H\})$ is a singleton. First note that for any $x\in \psi^{-1}(\{H\})$ and any $A\ni_{op} x$, we have $H\in \overline{\psi_{fib}(A)}$. In particular, since $\psi_{fib}(A)$ is open, we can for any $\epsilon > 0$ find $Hg\in (H\backslash G)\cap \psi_{fib}(A)$ with $d_G(Hg, H) < \epsilon$. Now if $x\neq y\in X$ satisfied $\psi(x) = \psi(y) = H$, we can find $A\ni_{op} x$, $B\ni_{op} y$, and $\epsilon > 0$ with $\overline{AU_\epsilon}\cap \overline{BU_\epsilon} = \emptyset$. This implies that $\overline{\psi_{fib}(A)U_\epsilon} \cap \overline{\psi_{fib}(B)U_\epsilon} = \emptyset$, a contradiction as $H$ is a member of this intersection.
		
In particular, by taking $H = \{1_G\}$, we see that $\Sa(G)$ is MHP.
We also have that $M(G)$ is MHP. There are two ways of seeing this. One is that $S_G(M(G))$ is a minimal flow mapping onto $M(G)$, so by uniqueness of $M(G)$ we have that $\pi_{M(G)}\colon S_G(M(G))\to M(G)$ is an isomorphism. The other way is to note that $M(G)$ is a retract of $\Sa(G)$ and observe that retracts of MHP flows are also MHP.

Also notice that since $H\backslash G$ is a dense $G$-invariant subspace of $\Sa(H\backslash G)$, then by the remark after Fact~\ref{Fact:UnivHP}, we have $\Sa(H\backslash G)\cong S_G(H\backslash G)$. When viewing $\Sa(H\backslash G)$ as a space of near ultrafilters, the following fact will be useful to keep in mind (see \cite{ZucThe}, Ch.\ 1).

\begin{fact}\mbox{}
	\label{Fact:SHGNults}
	If $X$ is a compact space and $f\colon H\backslash G\to X$ is a uniformly continuous function, then the unique continuous extension $f\colon \Sa(H\backslash G)\to X$ is defined by setting, for $p\in \Sa(H\backslash G)$ and $x\in X$, $f(p) = x$ iff $\{f^{-1}(U): U\ni_{op}x\}\subseteq p$. Given $p\in \Sa(H\backslash G)$, the existence of an $x\in X$ with this property is an easy consequence of compactness; the uniqueness of such an $x$ requires the uniform continuity of $f$.
\end{fact}

\subsubsection{Fra\"iss\'e expansion classes}
		
This example will not be needed in later sections and assumes some familiarity with \fr theory and expansion classes (see \cite{KPT} or \cite{ZucMetr}). Suppose $L$ is a countable language and $G = \mathrm{Aut}(\mathbf{K})$ for some \fr $L$-structure $\mathbf{K} = \flim{\mathcal{K}}$ with underlying set $\omega$. Let $\fin{\mathbf{K}}$ denote the collection of finite substructures of $\mathbf{K}$. Let $\mathcal{K}^*$ be a reasonable precompact expansion of $\mathcal{K}$ in a countable language $L^*\supseteq L$. Let $X_{L^*}$ denote the space of $L^*$-structures on $\omega$ endowed with the logic topology. We can endow $X_{L^*}$ with a continuous $G$-action, where for a structure $x\in X_{L^*}$, a relational symbol $R\in L^*$ of arity $n$, points $a_0,\ldots, a_{n-1}\in \omega$, and $g\in G$, we have 
$$R^{x\cdot g}(a_0,\ldots,a_{n-1}) \Leftrightarrow R^x(ga_0,\ldots, ga_{n-1}).$$
The definition is similar for function and constant symbols. We then form the $G$-flow 
$$X_{\mathcal{K}^*} = \{\mathbf{K}^*\in X_{L^*}: \mathbf{K}^*|_L = \mathbf{K} \text{ and } \mathbf{K}^*|_\mathbf{A} \in \mathcal{K}^* \text{ for every } \mathbf{A}\in \fin{\mathbf{K}}\}.$$
For $\mathbf{A}\in \fin{\mathbf{K}}$ and an expansion $\mathbf{A}^*\in \mathcal{K}^*$, a typical basic clopen neighborhood of $X_{\mathcal{K}^*}$ is given by 
$$N_{\mathbf{A}^*} = \{\mathbf{K}^*\in X_{\mathcal{K}^*}: \mathbf{K}^*|_\mathbf{A} = \mathbf{A}^*\}.$$

\begin{prop}
	\label{Prop:APImpliesMHP}
	Suppose $\mathcal{K}^*$ has the amalgamation property (AP). Then $X_{\mathcal{K}^*}$ is MHP. 
\end{prop}

\begin{proof}
	For $\mathbf{A}\in \fin{\mathbf{K}}$, write $U_\mathbf{A}\subseteq G$ for the pointwise stabilizer of $\mathbf{A}$. Then $U_\mathbf{A}\subseteq G$ is a clopen subgroup and a typical basic open neighborhood of $1_G\in G$. Let $W\subseteq X_{\mathcal{K}^*}$ be open. It suffices to show that $\overline{WU_\mathbf{A}}$ is clopen. To that end, we will show that for any $\mathbf{B}\in \fin{\mathbf{K}}$ with $\mathbf{A}\subseteq \mathbf{B}$ and any expansion $\mathbf{B}^*\in \mathcal{K}^*$, we have $\overline{N_{\mathbf{B}^*}U_\mathbf{A}} = N_{\mathbf{A}^*}$, where $\mathbf{A}^*$ is the expansion of $\mathbf{A}$ inherited from $\mathbf{B}^*$. The left-to-right inclusion is clear. For the other way, suppose $\mathbf{C}\in \fin{\mathbf{K}}$ is finite and $\mathbf{C}^*$ is an expansion so that $N_{\mathbf{C}^*}\subseteq N_{\mathbf{A}^*}$. By shrinking $N_{\mathbf{C}^*}$ if necessary, we may assume that $\mathbf{A}^*\subseteq \mathbf{C}^*$. Using the AP in $\mathcal{K}^*$, we can find $\mathbf{D}\in \fin{\mathbf{K}}$ and an expansion $\mathbf{D}^*$ so that $\mathbf{C}^*\subseteq \mathbf{D}^*$ and $f[\mathbf{B}^*]\subseteq \mathbf{D}^*$ for some $f\in \emb{\mathbf{B}^*, \mathbf{D}^*}$ with $f|_\mathbf{A} = 1_\mathbf{A}$. If $g\in G$ satisfies $g|_\mathbf{B} = f$, then $N_{\mathbf{B}^*}\cdot g^{-1}\cap N_{\mathbf{C}^*}\supseteq N_{\mathbf{D}^*}$, so is non-empty as desired.
\end{proof}
\vspace{2 mm}

We can provide a converse result as follows. Recall that a $G$-flow $X$ is \emph{topologically transitive} if for every $A, B\subseteq_{op} X$, there is $g\in G$ with $Ag\cap B\neq \emptyset$. 
\vspace{2 mm}

\begin{prop}
	\label{Prop:MHPImpliesAP}
	Suppose $X$ is a metrizable MHP $G$-flow. Then there is an reasonable, precompact expansion class $\mathcal{K}^*$ with the AP so that $X\cong X_{\mathcal{K}^*}$. If $X$ is also topologically transitive, then we can take the class $\mathcal{K}^*$ to be  Fra\"iss\'e.
\end{prop} 
\vspace{2 mm}

\begin{proof}
Let $\mathbf{A}\in \fin{\mathbf{K}}$. Then if $W\subseteq X$ is open, the equality $\overline{W\cdot U_\mathbf{A}}\cdot U_\mathbf{A} = \overline{W\cdot U_\mathbf{A}}$ and MHP show that $\overline{W\cdot U_\mathbf{A}}$ is clopen. Call a clopen set $Y\subseteq X$ \emph{$U_\mathbf{A}$-clopen} if $Y\cdot U_\mathbf{A} = Y$; the collection $\mathcal{B}(\mathbf{A})$ of $U_\mathbf{A}$-clopen sets forms an algebra. 

Suppose $\mathcal{B}(\mathbf{A})$ were infinite. Then we could find $\{Y_n: n< \omega\}$ a collection of pairwise disjoint members of $\mathcal{B}(\mathbf{A})$. For $S\subseteq \omega$, write $Y_S = \bigcup_{n\in S} Y_n$. Then if $y\in \overline{Y_S}$, we have $y\in \mathrm{int}(\overline{Y_S\cdot U_\mathbf{A}}) = \mathrm{int}(\overline{Y_S})$, i.e.\ the set $\overline{Y_S}$ is clopen. It follows that for $S,T\subseteq \omega$ disjoint, we have $\overline{Y_S}\cap \overline{Y_T} = \emptyset$. It follows that if $p_n\in Y_n$ for each $n< \omega$, then $\overline{\{p_n: n< \omega\}}$ is isomorphic to $\beta \mathbb{N}$, contradicting our assumption that $X$ is metrizable. Hence $\mathcal{B}(\mathbf{A})$ is finite, hence atomic. Let $\mathrm{Atoms}(\mathbf{A})\subseteq \mathcal{B}(\mathbf{A})$ denote the atoms.

To each $\mathbf{A}\in \fin{\mathbf{K}}$, we can view $\mathrm{Atoms}(\mathbf{A})$ as a set of ``expansions'' of $\mathbf{A}$. Suppose $\mathbf{B}\in \fin{\mathbf{K}}$, $Z\in \mathrm{Atoms}(\mathbf{B})$, and $f\colon \mathbf{A}\to \mathbf{B}$ is an embedding. We need to determine which expansion of $\mathbf{A}$ is induced by $f$ when we expand $\mathbf{B}$ using $Z$. We do this as follows: first find $g\in G$ with $g|_\mathbf{A} = f$. We will argue that $Zg$ is contained in some $U_\mathbf{A}$-atom, and that this does not depend on the $g$ we chose. So suppose $W$ is $U_\mathbf{A}$-clopen. By choice of $g$, We have $g^{-1}U_\mathbf{B}g\subseteq U_\mathbf{A}$, so $Wg^{-1}U_\mathbf{B} = Wg^{-1}$. This shows that $Wg^{-1}$ is $U_\mathbf{B}$-clopen. Therefore if $Zg\cap W\neq \emptyset$, then $Z\cap Wg^{-1}\neq \emptyset$, so $Z\subseteq Wg^{-1}$ and $Zg\subseteq W$. It follows that $Zg$ is contained in some $U_\mathbf{A}$-atom, say $Y$. If $h\in G$ also satisfies $h|_\mathbf{A} = f$, then $g^{-1}h\in U_\mathbf{A}$, so $Zg(g^{-1}h)\subseteq Y$ as well. Therefore if $\mathbf{B}^Z$ is the corresponding expansion of $\mathbf{B}$, we declare that $\mathbf{A}^Y$ is the expansion that $\mathbf{A}$ inherits from $\mathbf{B}^Z$ along the map $f\colon \mathbf{A}\to \mathbf{B}$. All of this can be coded by adding countably many new relational symbols to $L$, producing a language $L^*\supseteq L$ and a reasonable precompact expansion class $\mathcal{K}^*$ of $\mathcal{K}$. 

For each $\mathbf{A}\in \fin{\mathbf{K}}$, the set $\mathrm{Atoms}(\mathbf{A})$ is a finite clopen partition of the space $X$. If $x\in X$, it follows that $x\in Y$ for exactly one $U_\mathbf{A}$-atom for each $\mathbf{A}\in \fin{\mathbf{K}}$, giving rise to a surjective $G$-map $\phi\colon X\to X_{\mathcal{K}^*}$. If $x\neq y\in X$, then by continuity of the action, we can find $V\ni_{op} x$, $W\ni_{op} y$, and $\mathbf{A}\in \fin{\mathbf{K}}$ with $\overline{VU_\mathbf{A}}\cap \overline{WU_\mathbf{A}} = \emptyset$, showing that $\phi$ is injective, hence an isomorphism. 

To show that this expansion class has the AP, suppose we have $\mathbf{A}, \mathbf{B}, \mathbf{C}\in \fin{\mathbf{K}}$ with $\mathbf{A}\subseteq \mathbf{B}$ and $\mathbf{A}\subseteq \mathbf{C}$. Let $Y_A, Y_B, Y_C\subseteq X$ be clopen atomic sets for $U_\mathbf{A}, U_\mathbf{B}, U_\mathbf{C}$, respectively, with $Y_B\subseteq Y_A$ and $Y_C\subseteq Y_A$. Since $Y_A$ is a $U_\mathbf{A}$-atom, the action of $U_\mathbf{A}$ on $Y_A$ is topologically transitive, so we can find $g\in U_\mathbf{A}$ with $Y_Cg \cap Y_B\neq \emptyset$. We can then find some suitably large finite $\mathbf{D}\subseteq \mathbf{K}$ so that for some $U_\mathbf{D}$-atom $Y_D$ we have $Y_D\subseteq Y_Cg\cap Y_B$. By enlarging $\mathbf{D}$ more if needed, we can assume that $\mathbf{B}\subseteq \mathbf{D}$ and $g^{-1}[\mathbf{C}]\subseteq \mathbf{D}$. It follows that $i_\mathbf{B}\colon \mathbf{B}^{Y_B}\to \mathbf{D}^{Y_D}$ and $(g^{-1})|_\mathbf{C}\colon \mathbf{C}^{Y_C}\to \mathbf{D}^{Y_D}$ amalgamate the maps $i_\mathbf{A}\colon \mathbf{A}^{Y_A}\to \mathbf{B}^{Y_B}$ and $i_\mathbf{A}\colon \mathbf{A}^{Y_A}\to \mathbf{C}^{Y_C}$. 

A similar argument shows that if $X$ is topologically transitive, then the expansion $\mathcal{K}^*$ that we constructed above will have the joint embedding property (JEP) as well. 
\end{proof}

In the case that $X$ is topologically transitive, MHP, but not necessarily metrizable, two important cases emerge. Either for every finite $\mathbf{A}\subseteq \mathbf{K}$, the algebra of $U_\mathbf{A}$-clopen sets is atomic, or this fails for some $\mathbf{A}$; the equivalent conditions of  Theorem~\ref{IntroThm:MHP} correspond to the first case. 
\vspace{2 mm}

\section{Topometrics on MHP flows}

For the rest of the section, fix an MHP flow $(X, \tau)$, where $\tau$ is the compact topology on $X$. Our goal is to endow $X$ with a canonical topometric structure. This has been done in the case of $\Sa(G)$ in \cite{BYMT}, where they use the following definition. Before stating the definition, we note that if $f\colon G\to [0,1]$ is left-uniformly continuous, we can continuously extend it to $\Sa(G)$, and we will also use $f\colon \Sa(G)\to [0,1]$ to denote this extension.
\vspace{2 mm}

\begin{defin}
	\label{Def:TopometricSG}
	Given $p, q\in \Sa(G)$, we set 
	$$\partial(p, q) = \sup(|f(p) - f(q)|:\, f\colon G\to [0,1] \text{ $1$-Lipschitz}).$$
\end{defin}
\vspace{2 mm}

Notice that if $f\colon G\to [0,1]$ is $1$-Lipschitz and we continuously extend to $\Sa(G)$, then $f$ has the following property, which we define more generally.
\vspace{2 mm}

\begin{defin}
	\label{Def:OrbitLip}
	Let $X$ be a $G$-flow. A function $f\in C(X, [0,1])$ is called \emph{orbit Lipschitz} if whenever $x\in X$ and $g\in G$, we have 
	$$|f(x) - f(xg)| \leq d_G(1_G, g).$$
	We write $C_{OL}(X, [0,1])$ for the collection of orbit Lipschitz functions.
\end{defin}
\vspace{2 mm}

Eventually, we will show that the analogue of Definition \ref{Def:TopometricSG} with $1$-Lipschitz replaced by orbit Lipschitz provides the MHP flow $X$ with a topometric structure. The problem is that a priori, we do not know whether $X$ has any non-constant orbit Lipschitz functions. Therefore we start with an entirely different definition of the topometric structure, then use Fact~\ref{Fact:ExistLipschitz} to produce an ample supply of continuous Lipschitz functions, which will turn out to be precisely the orbit Lipschitz functions. 
\vspace{2 mm}

\begin{defin}
	\label{Def:TopometricGeneral}
	Given $x, y\in X$ and $c\geq 0$, we define $\partial(x, y) \leq c$ iff any of the following four equivalent items hold.
	\vspace{-2 mm}
	
	\begin{enumerate}
		\item 
		Whenever $A\subseteq_{op} X$ with $x\in \overline{A}$ and $\epsilon > 0$, we have $y\in \mathrm{int}(\overline{AU_{c+\epsilon}})$.
		\item 
		Whenever $A\subseteq_{op} X$ with $x\in \overline{A}$ and $\epsilon > 0$, we have $y\in \overline{AU_{c+\epsilon}}$.
		\item 
		Whenever $A\ni_{op} x$ and $\epsilon > 0$, we have $y\in \mathrm{int}(\overline{AU_{c+\epsilon}})$.
		\item 
		Whenever $A\ni_{op} x$ and $\epsilon > 0$, we have $y\in \overline{AU_{c+\epsilon}}$.
	\end{enumerate}
\end{defin}

\begin{rem}
	The directions $(1)\Rightarrow (2)\Rightarrow (4)$ as well as $(1)\Rightarrow (3)\Rightarrow (4)$ are clear. Suppose $(4)$ holds, and let $A\subseteq_{op} X$ with $x\in \overline{A}$. Also fix $\epsilon > 0$. Then as $X$ is MHP, we have $x\in \mathrm{int}(\overline{AU_{\epsilon}})$. By $(4)$, we have $y\in \overline{\mathrm{int}(\overline{AU_{\epsilon}})U_{c+\epsilon}} \subseteq \overline{AU_{c+2\epsilon}}$. Using MHP once more, we obtain $y\in \mathrm{int}(\overline{AU_{c+3\epsilon}})$, showing that $(1)$ holds.
\end{rem}
\vspace{2 mm}

\begin{prop}
	\label{Prop:IsAMetric}
	The function $\partial$ from Definition \ref{Def:TopometricGeneral} is a topometric on $X$. 
\end{prop}

\begin{proof}
	Suppose $x, y\in X$ have $\partial(x, y) = 0$. If $A\ni_{op} x$, we can find $B\ni_{op} x$ and $\epsilon > 0$ with $\overline{BU_\epsilon}\subseteq A$. So in particular $y\in A$, so $x = y$.
	
	Suppose $\partial(x, y)\leq c$ for some $c\geq 0$ towards showing that $\partial(y, x)\leq c$. Let $B\ni_{op} y$ and $\epsilon > 0$. Notice that if $A\ni_{op} x$, then $AU_{c+\epsilon}\cap B\neq \emptyset$. So also $A\cap BU_{c+\epsilon}\neq \emptyset$. It follows that $x\in \overline{BU_{c+\epsilon}}$.
	
	Now suppose $\partial(x, y) \leq c$ and $\partial(y, z) \leq d$. Fix $A\ni_{op} x$ and $\epsilon > 0$. Then $AU_{c+\epsilon}$ is open with $y\in \overline{AU_{c+\epsilon}}$. We then have $z\in \overline{AU_{c+d+2\epsilon}}$, showing that $\partial(x, z) \leq c+d$ as desired.
	
	Having shown that $\partial$ is a metric on $X$, we now show that it is $\tau$-lsc. Fix $c\geq 0$, and let $x_i\to x$ and $y_i\to y$ be nets with $\partial(x_i, y_i)\leq c$. Let $A\ni_{op} x$, and fix $\epsilon > 0$. Then for a tail of $x_i$, we also have $x_i\in A$, implying that $y_i\in \overline{AU_{c+\epsilon}}$. So $y\in \overline{AU_{c+\epsilon}}$, and by MHP, $y\in \mathrm{int}(\overline{AU_{c+2\epsilon}})$. It follows that $\partial(x, y)\leq c$.
\end{proof}
\vspace{2 mm}

\begin{rem}
	If $G$ is a discrete group and $X$ is an MHP $G$-flow, then $\partial$ is just the discrete metric on $X$. If $G$ is locally compact and $c\geq 0$ is small enough so that $U_{c+\epsilon}\subseteq G$ is precompact for some $\epsilon > 0$, then given an MHP $G$-flow $X$ and $x, y\in X$, we have $\partial(x, y)\leq c$ iff there is $g\in G$ with $d(g, 1_G) \leq c$ and $xg = y$. Hence topometric structures on MHP flows are most interesting when $G$ is not locally compact.
\end{rem}
\vspace{2 mm}

\begin{rem}
	Suppose $H\subseteq G$ is a closed subgroup, and form $\Sa(H\backslash G)$. On the orbit $H\b G\subseteq \Sa(H\backslash G)$, the metric $\partial$ coincides with the metric $d$. In particular, this is true for $G\subseteq \Sa(G)$. This will be easiest to see by using Corollary~\ref{Cor:Topometric}.
\end{rem}
\vspace{2 mm}

\begin{rem}
	Suppose $\mathbf{K} = \flim{\mathcal{K}}$ is a \fr structure with $G = \aut{\mathbf{K}}$. Write $\mathbf{K} = \bigcup_{n\geq 1} \mathbf{A}_n$ as an increasing union of finite structures. A compatible left-invariant metric $d$ on $G$ is given by $d(g, h) \leq 1/n$ iff $g|_{\mathbf{A}_n} = h|_{\mathbf{A}_n}$. Write $V_n = \{g\in G: g|_{\mathbf{A}_n} = \mathrm{id}_{\mathbf{A}_n}\}$, and notice that for any suitably small $\epsilon > 0$, we have $V_n = U_{1/n + \epsilon}$.
	
	Now suppose $X$ is an MHP $G$-flow. As in the discussion before Proposition~\ref{Prop:MHPImpliesAP}, let $\mathcal{B}_n$ be the Boolean algebra of $V_n$-clopen subsets of $X$. As we make no metrizability assumption here, $\mathcal{B}_n$ may be infinite. However, if $\mathcal{Y}\subseteq \mathcal{B}$ and we set $Y = \bigcup \mathcal{Y}$, we see that for $y\in \overline{Y}$, we have $y\in \mathrm{int}(\overline{Y\cdot V_n}) = \mathrm{int}(\overline{Y})$. In particular, $\overline{Y}\in \mathcal{B}_n$. Setting $\bigvee \mathcal{Y} = \overline{Y}$, we see that $\mathcal{B}_n$ is a complete Boolean algebra. Let $X_n = \mathrm{St}(\mathcal{B}_n)$ be the Stone space. Then $X\cong \varprojlim X_n$, and given $x = (x_n)_n$ and $y = (y_n)_n$ in $X$, we have that $\partial(x, y) \leq 1/n$ iff $x_n = y_n$. To see this, first suppose $x_n\neq y_n$, and find some $A\in \mathcal{B}_n$ with $x\in A$ and $y\not\in A$. But since $A = AU_{1/n+\epsilon} = \overline{AU_{1/n+\epsilon}}$, we have $\partial(x, y) > 1/n$ by item $(4)$ of Definition~\ref{Def:TopometricGeneral}. In the other direction, suppose $x_n = y_n$. Then if $A\ni_{op} x$, we have $\mathrm{int}(\overline{AU_{1/n+\epsilon}})\in \mathcal{B}_n$, hence $y\in \mathrm{int}(\overline{AU_{1/n+\epsilon}})$. Therefore $\partial(x, y)\leq 1/n$ by item $(3)$ of Definition~\ref{Def:TopometricGeneral}. 
\end{rem}
\vspace{2 mm}

We next investigate how this topometric structure interacts with the $G$-flow structure. Not only is this a canonical topometric to place on an MHP flow $X$, but it will also behave well when comparing different MHP flows. When discussing multiple MHP flows $X$, $Y$, etc., we write $\partial_X$, $\partial_Y$, etc.\ to refer to the topometric structure on each flow.
\vspace{2 mm}

\begin{prop}
	\label{Prop:TopometricDynamics}
	Let $X$ and $Y$ be MHP flows endowed with the topometric structure from Definition~\ref{Def:TopometricGeneral}. 
	\vspace{-2 mm}
	
	\begin{enumerate}
		\item 
		If $x\in X$ and $g\in G$, then $\partial(x, xg) \leq d_G(1_G, g)$.
		\item 
		For each $g\in G$, the map $\rho_g\colon (X, \partial)\to (X, \partial)$ given by $\rho_g(x) = xg$ is uniformly continuous.
		\item
		If $\phi\colon X\to Y$ is a $G$-map, then $\phi$ is metrically non-expansive, i.e.\ for any $x, y\in X$, we have $\partial_Y(\phi(x), \phi(y))\leq \partial_X(x, y)$.
	\end{enumerate}
\end{prop}

\begin{proof}
	For item (1), write $c = d_G(1_G, g)$. Then for any $\epsilon > 0$, we have $g\in U_{c+\epsilon}$. Hence if $A\ni_{op} x$, we have $xg\in AU_{c+\epsilon}$.
	
	For item (2), fix $c > 0$. Find $d > 0$ so that $g^{-1}U_dg\subseteq U_c$. Now suppose $x, y\in X$ satisfy $\partial(x, y) < d$. Let $A\ni_{op} xg$, and fix $\epsilon > 0$. Then $Ag^{-1}\ni_{op} x$, so $y\in \overline{Ag^{-1}U_d}$. It follows that $yg\in \overline{AU_c}$, so $\partial(x, y)\leq c$.
	
	For item (3), write $c = \partial_X(x, y)$, and let $B\ni_{op} \phi(x)$. Then $x\in \phi^{-1}(B)$, so we have $y\in \overline{\phi^{-1}(B)U_{c+\epsilon}} = \overline{\phi^{-1}(BU_{c+\epsilon})}\subseteq \phi^{-1}(\overline{BU_{c+\epsilon}})$ for any $\epsilon > 0$. So $\phi(y)\in \overline{BU_{c+\epsilon}}$ as desired. 
\end{proof}
\vspace{2 mm}

\begin{rem}
	Notice that item $(3)$ shows that if $M\subseteq \Sa(G)$ is a minimal subflow, then the topometric structure computed internally in $M$ is the same as the topometric structure inherited from $\Sa(G)$. This is because $M$ is a retract of $\Sa(G)$. 
\end{rem}
\vspace{2 mm}

Denote by $C_L(X, [0,1])$ the collection of continuous, $1$-Lipschitz functions from $X$ to $[0,1]$. The next proposition along with Fact~\ref{Fact:ExistLipschitz} will give us Corollary~\ref{Cor:Topometric}, the analogue of Definition~\ref{Def:TopometricSG} for any MHP flow. 
\vspace{2 mm}

\begin{prop}
	\label{Prop:LipschitzTopometric}
	$C_L(X, [0,1]) = C_{OL}(X, [0,1])$
\end{prop}

\begin{proof}
	First suppose $f\in C_L(X, [0,1])$. Then since for any $p\in X$ and $g\in G$, we have $\partial(p, pg) \leq d(1_G, g)$, we see that $f\in C_{OL}(X, [0,1])$.
	
	Now suppose $f\in C_{OL}(X, [0,1])$, and fix $p, q\in X$. Suppose $\partial(p, q)\leq c$, and let $\epsilon > 0$. Find $A\ni_{op} p$ so that $|f(p')-f(p)| < \epsilon$ for $p'\in A$. Then $q\in \overline{AU_{c+\epsilon}}$, so find $p_i\in A$ and $g_i\in U_{c+\epsilon}$ with $p_ig_i\to q$. As $f$ is orbit Lipschitz, we have $|f(p_ig_i) - f(p)| < c+ 2\epsilon$. As $\epsilon > 0$ is arbitrary, we have $|f(q)-f(p)| \leq c$ as desired.	
\end{proof}
\vspace{2 mm}

\begin{cor}
	\label{Cor:Topometric}
	Let $x, y\in X$. Then $\partial(x, y) = \sup\{|f(x) - f(y)|: f\in C_{OL}(X, [0,1])\}$.
\end{cor}
\vspace{2 mm}

We end the section by proving that the topometric space $(X, \tau, \partial)$ is adequate. For the proof, it will be easier to work with closed sets rather than open sets. Given $K\subseteq X$ and $c > 0$, we write $K(-c) := X\setminus((X\setminus K)(c)) = \{x\in X: x(c)\subseteq K\}$. So a topometric space $(X, \tau, \partial)$ is adequate if for every closed $K\subseteq X$ and every $c > 0$, we have $K(-c)$ closed.
\vspace{2 mm}

\begin{theorem}
	\label{Thm:Adequate}
	The topometric space $(X, \tau, \partial)$ is adequate.
\end{theorem}

\begin{proof}
	Fix $K\subseteq X$ closed. We show that the set $K(-c)$ is also closed. Write $K = \bigcap_i K_i$ with each $K_i$ a regular closed set. Then $K(-c) = \bigcap_i K_i(-c)$. So it suffices to prove the theorem in the case that $K$ is regular closed (we will only need this at the very end). For such $K$, we will show that
	$$K(-c) = \bigcap_{\substack{\epsilon > 0\\ r < c}} X\setminus \mathrm{int}\left(\overline{(X\setminus \overline{KU_\epsilon})U_r}\right).$$
	Suppose $p\in X$ is not in the left hand side. Then there is $q\in X\setminus K$ with $\partial(p,q) < c$. Given $r$ with $\partial(p, q) < r < c$, then for every $A\ni_{op} q$, we have $p\in \mathrm{int}(\overline{AU_r})$. Now for some suitably small $\epsilon > 0$, we have $q\in X\setminus \overline{KU_\epsilon}$. Taking $A = X\setminus \overline{KU_\epsilon}$, we see that $p$ is not in the right hand side. 
	
	Now suppose $p\in X$ is not in the right hand side as witnessed by $\epsilon > 0$ and $r < c$. In particular, we have $p\in \overline{(X\setminus \overline{KU_\epsilon})U_r}$. Let $A\ni_{op} p$. Then $A\cap (X\setminus \overline{KU_\epsilon})U_r\neq \emptyset$. It follows that $AU_r\cap (X\setminus \overline{KU_\epsilon}) \neq \emptyset$. Therefore we have
	$$\left(\overline{X\setminus \overline{KU_\epsilon}}\right)\cap \left(\bigcap \{\overline{AU_r}: A\ni_{op} p\}\right)\neq \emptyset.$$
	Fix some $q$ from this set. It follows that $\partial(p, q)\leq r < c$. To see that $q\not\in K$, notice that for any $x\in K$, we have $x\in \overline{\mathrm{int}(K)}$, so we have $x\in \mathrm{int}(\overline{\mathrm{int}(K)U_\epsilon}) = \mathrm{int}(\overline{KU_\epsilon})$. 
\end{proof}
\vspace{2 mm}

\section{Comeager orbits in MHP flows}

We continue with most of the notation of the previous section. In particular, $G$ is a Polish group, and $(X, \tau, \partial)$ is an MHP $G$-flow endowed with the topometric structure from Definition~\ref{Def:TopometricGeneral}. In this section, we undertake a deeper study of the interaction between the topology $\tau$ and the metric $\partial$, connecting this to various properties that the $G$-flow $X$ might enjoy. The main theorem is Theorem~\ref{IntroThm:MHP}, which gives a complete characterization of when an MHP flow has a comeager orbit.
\vspace{2 mm}

\begin{defin}
	\label{Def:CompatiblePoint}
	Let $p\in X$. We say that $\partial$ is \emph{compatible} at $p$ or that $p$ is a \emph{compatiblity point} if for every $c> 0$, we have $p\in \mathrm{int}(p(c))$. 
\end{defin}
\vspace{2 mm}

Compatibility points are precisely the points in $X$ where the topologies given by $\tau$ and $\partial$ coincide. This is a notion which has been studied in the context of continuous logic, especially in regards to type spaces and the omitting types theorem (see \cite{BBHU}, Ch.\ 12). We can now generalize one of the key theorems from \cite{BYMT}. We will repeatedly use the fact that if $A, B\subseteq_{op} X$ with $A\cap B = \emptyset$, then $\mathrm{int}(\overline{A})\cap \mathrm{int}(\overline{B}) = \emptyset$.
\vspace{2 mm}

\begin{lemma}
	\label{Lem:SymmetricTopo}
	Suppose $x, y\in X$ satisfy $\partial(x, y) > 2c$. Then there are $A\ni_{op} x$ and $B\ni_{op} y$ with $\overline{AU_c}\cap \overline{BU_c} = \emptyset$.
\end{lemma}

\begin{proof}
	We can find $A\ni_{op} x$ and $\epsilon > 0$ with $y\not\in\overline{AU_{2c+2\epsilon}}$. Setting $B = X\setminus \overline{AU_{2c+2\epsilon}}$, we have that $\overline{AU_c}\cap \overline{BU_c} = \emptyset$ as desired. Indeed if $x\in \overline{AU_c}\cap \overline{BU_c}$, then by MHP $x\in \mathrm{int}(\overline{AU_{c+\epsilon}})\cap \mathrm{int}(\overline{BU_{c+\epsilon}})$, contradicting that $AU_{c+\epsilon}\cap BU_{c+\epsilon} = \emptyset$.
\end{proof}
\vspace{2 mm}

\begin{theorem}
	\label{Thm:Metrizability}
	$(X, \tau)$ is metrizable iff $\partial$ is a compatible metric for $\tau$, i.e.\ iff $\partial$ is compatible at every point in $X$. Furthermore, if $(X, \tau)$ is not metrizable, then $X$ embeds a copy of $\beta\omega$, the space of ultrafilters on $\omega$. 
\end{theorem}

\begin{proof}
	One direction is clear, so suppose $\partial$ generates a strictly finer topology than $\tau$. In particular, $(X, \partial)$ is not compact, so find $c> 0$ and an infinite $Y\subseteq X$ with $\partial(x, y) > 2c$ for any $x\neq y\in Y$. 
	
	We will inductively define infinite $Y_n\subseteq Y$,  $x_n \in Y_n$, and $A_n\ni_{op} x_n$ for each $n< \omega$. We will ensure that the following all hold.
	\vspace{-2 mm}
	
	\begin{enumerate}
		\item 
		$Y_{n+1}\subseteq Y_n$ for every $n< \omega$.
		\item 
		$Y_n\cap \overline{A_kU_c} = \emptyset$ for every $k< n< \omega$. 
		\item 
		$A_n\cap A_kU_c = \emptyset$ for each $k< n< \omega$
	\end{enumerate}
	
	Set $Y_0 = Y$. Suppose $Y_0,\ldots,Y_n$, $x_0,\ldots x_{n-1}$, and $A_0\ldots A_{n-1}$ have been chosen. Pick $x\neq y\in Y_n$, and use Lemma~\ref{Lem:SymmetricTopo} to find $A\ni_{op} x$ and $B\ni_{op} y$ with $\overline{AU_c}\cap \overline{BU_c} = \emptyset$. We also demand by shrinking $A$ and $B$ if needed that $A\cap A_kU_c = \emptyset$ and $B\cap A_kU_c= \emptyset$ for every $k< n$; this is possible by item (2). Now at least one of $Y_n\setminus \overline{AU_c}$ or $Y_n\setminus \overline{BU_c}$ is infinite, without loss of generality the former. Set $Y_{n+1} = Y_n\setminus \overline{AU_c}$, $x_n = x$, and $A_n = A$.
	
	Having completed the inductive construction, define $\phi\colon \beta\omega\to X$ to be the continuous extension of the map $\phi(n) = x_n$. We show that $\phi$ is injective. If $S\subseteq \omega$, set $A_S = \bigcup_{n\in S} A_n$. It is enough to show that if $S, T\subseteq \omega$ with $S\cap T = \emptyset$, then $\overline{A_S}\cap \overline{A_T} = \emptyset$. To see why this is, note that $\overline{A_S}\subseteq \mathrm{int}(\overline{A_SU_{c/2}})$, likewise for $A_T$, and that $A_SU_{c/2}\cap A_TU_{c/2} = \emptyset$.
\end{proof}
\vspace{2 mm}

Next we investigate what happens when some, but not all, points in $X$ are compatibility points. We remind the reader that the topometric space $(X, \tau, \partial)$ was proven in Theorem~\ref{Thm:Adequate} to be adequate.
\vspace{2 mm}

\begin{lemma}\mbox{}
	\label{Lem:Compatibility}
	\vspace{-2 mm}
	
	\begin{enumerate}
		\item 
		Let $Y\subseteq X$ denote the set of compatibility points. Then $Y$ is $G$-invariant, $\partial$-closed, and topologically $G_\delta$.
		\item 
		Suppose $x\in X$ is not a compatibility point. Then there is $c> 0$ so that $\mathrm{int}(x(c)) = \emptyset$.
	\end{enumerate}
\end{lemma}

\begin{proof}\mbox{}
	\vspace{-2 mm}
	
	\begin{enumerate}
		\item 
		That $Y$ is $G$-invariant follows from item (2) of Proposition~\ref{Prop:TopometricDynamics}. To show $Y$ is  $\partial$-closed, let $y_n\xrightarrow{\partial} y$, and fix $c > 0$. Then for some $n< \omega$, we have $y_n(c/2)\subseteq y(c)$. By assumption, $y_n\in \mathrm{int}(y_n(c/2))$, so in particular, we have $\mathrm{int}(y(c))\neq \emptyset$. Using adequacy, we have $y\in (\mathrm{int}(y(c)))(c)\subseteq \mathrm{int}(y(2c))$. Lastly, to show that $Y$ is $G_\delta$, let $Y_c = \bigcup_{y\in Y} \mathrm{int}(y(c))$. Then $Y_c\subseteq X$ is open with $Y = \bigcap_{c > 0} Y_c$. 
		\item 
		Suppose $x\in X$ is a point with $\mathrm{int}(x(c))\neq \emptyset$ for every $c> 0$. Then $x$ is a compatibility point, as by adequacy, we have $x\in (\mathrm{int}(x(c)))(c)\subseteq \mathrm{int}(x(2c))$.
		\qedhere
	\end{enumerate}
\end{proof}
\vspace{2 mm}

\begin{theorem}
	\label{IntroThm:MHP}
	The following are equivalent.
	\vspace{-2 mm}
	
	\begin{enumerate}
		\item 
		$X$ has a compatibility point with dense orbit.
		\item 
		The set $Y\subseteq X$ of compatibility points is comeager, Polish, and $G$ acts on $Y$ topologically transitively. 
		\item 
		$X$ has a comeager orbit.
		\item 
		$X\cong \Sa(H\backslash G)$ for some closed subgroup $H\subseteq G$.
	\end{enumerate}
\end{theorem}

\begin{proof}\mbox{}
	
	$(1)\Rightarrow (2)$ Letting $Y\subseteq X$ denote the set of compatibility points, item (1) of Lemma~\ref{Lem:Compatibility} shows us that $Y$ is $G$-invariant, $\partial$-closed, and $G_\delta$. By (1), $Y\subseteq X$ is dense. As $(Y, \tau)$ and $(Y, \partial)$ are homeomorphic, we see that $(Y, \tau)$ is separable and that $\partial$ is a compatible complete metric, hence $Y$ is Polish. As $Y$ contains a dense orbit, the action of $G$ on $Y$ is topologically transitive. 
	
	$(2)\Rightarrow (3)$ It is enough to show that $Y$ has a comeager orbit. We mostly follow the proof from~\cite{BYMT}, with a few differences to adapt to our more general setting. Using a criterion due to Rosendal (see~\cite{BYMT} for a proof of the criterion), we need to show that for every $\epsilon > 0$ and every $A\subseteq_{op} Y$, there is $B\subseteq_{op} A$ so that the local action of $U_\epsilon$ on $B$ is topologically transitive. To that end, let $B\subseteq_{op} Y$ be any open set of $\partial$-diameter less than $\epsilon$; any $A\subseteq_{op} Y$ will contain such a $B$ since $(Y, \partial)$ and $(Y, \tau)$ are homeomorphic. Fix $C_0, C_1\subseteq_{op} B$. If $p_0\in C_0$ and $p_1\in C_1$, then $\partial(p_0, p_1)< \epsilon$. So $p_1\in \overline{C_0U_\epsilon}$. In particular, $C_0U_\epsilon\cap C_1\neq \emptyset$ as desired.
	
	$(3)\Rightarrow (4)$. Let $Z\subseteq X$ denote the comeager orbit, and pick $p\in Z$. Let $H = \mathrm{Stab}(p)$. By the Effros theorem, we have that $Z\cong H\backslash G$ as $G$-spaces. So also $S_G(Z)\cong S_G(H\b G) \cong \Sa(H\b G)$. But since $Z\subseteq X$ is dense, we have $S_G(Z)\cong S_G(X) \cong X$.
	
	$(4)\Rightarrow (1)$ We will make use of Fact~\ref{Fact:SHGNults}. For each $\epsilon > 0$, we have that $C_{HU_\epsilon}\subseteq \Sa(H\backslash G)$ is a neighborhood of $H$.  Let $f\colon \Sa(H\backslash G)\to [0,1]$ be continuous and orbit-Lipschitz. In particular, $f|_{H\backslash G}$ is $1$-Lipschitz. So if $p\in C_{HU\epsilon}$, we have $|f(p) - f(H)| \leq \epsilon$. Therefore $C_{HU_\epsilon}\subseteq H(2\epsilon)$, so $H$ is a compatibility point in $\Sa(H\backslash G)$. 
\end{proof} 
\vspace{2 mm}

\section{More on Samuel compactifications}
Given item $(4)$ in Theorem~\ref{IntroThm:MHP}, let us spend some time to develop a more detailed understanding of the topometric $G$-space $\Sa(H\backslash G)$, which we continue to view as a space of near ultrafilters. We first consider the left completion $\widehat{H\backslash G}$. Notice that if $f\colon H\backslash G\to X$ is a uniformly continuous function with $X$ a complete uniform space, then $f$ continuously extends to $\widehat{H\backslash G}$. In particular, by considering the inclusion $i\colon H\backslash G\hookrightarrow \Sa(H\backslash G)$, we obtain a continuous map from $\widehat{H\backslash G}$ to $\Sa(H\backslash G)$. This map turns out to be an embedding, and we will identify $\widehat{H\backslash G}$ with its image in $\Sa(H\backslash G)$. We have the following fact (see \cite{ZucThe}, Ch.\ 1.2).
\vspace{2 mm}

\begin{fact}
	\label{Fact:CompletionInSamuel}
	Given $p\in \Sa(H\backslash G)$, we have $p\in \widehat{H\backslash G}$ iff for every $\epsilon > 0$, there is $A\subseteq_{op} H\backslash G$ of diameter less than $\epsilon$ with $A\in p$.
\end{fact}
\vspace{2 mm}

From the proof of Theorem~\ref{IntroThm:MHP}, we know that $H\in \Sa(H\backslash G)$ is a compatibility point. As $H$ has dense orbit in $\Sa(H\backslash G)$, and since the topology and the metric coincide on the set of compatibility points, we see that $H$ has a $\partial$-dense orbit in the set of compatibility points. Since $\partial$ and $d$ coincide on $H\backslash G$, we obtain the following.
\vspace{2 mm}

\begin{prop}
	\label{Prop:CompPointsInSamuel}
	In $\Sa(H\backslash G)$, the set of compatibility points is precisely $\widehat{H\backslash G}$.
\end{prop}
\vspace{2 mm}

In particular, by Theorem~\ref{IntroThm:MHP}, we have that $\widehat{H\backslash G}\subseteq \Sa(H\backslash G)$ is comeager. As $H\backslash G\subseteq \widehat{H\backslash G}$ is comeager, we obtain the following.
\vspace{2 mm}

\begin{prop}
	\label{Prop:ComeagerOrbitSamuel}
	In $\Sa(H\backslash G)$, the orbit $H\backslash G\subseteq \Sa(H\backslash G)$ is comeager.
\end{prop}
\vspace{1 mm}

\begin{rem}
	This proposition is really a statement about topology rather than dynamics. Whenever $(X, d)$ is a Polish metric space and $S(X)$ is the Samuel compactification of $X$ with its metric uniformity, then $X\subseteq S(X)$ is comeager.
\end{rem}
\vspace{2 mm}

We now take some time to understand the canonical $G$-map $\pi\colon \Sa(G)\to \Sa(H\backslash G)$. To do this, we first need to understand how near ultrafilters on $H$ interact with those on $G$. Let $p\in \Sa(H)$. Then if $A\in p$ and $\epsilon > 0$, we have $AU_\epsilon \subseteq_{op} G$, and the collection $\{AU_\epsilon: A\in p, \epsilon > 0\}$ extends to a unique near ultrafilter in $\Sa(G)$. This gives rise to an embedding $i\colon \Sa(H)\hookrightarrow \Sa(G)$. More explicitly, given $p\in G$, we set
$$i(p) = \{B\subseteq_{op} G: B\cap AU_\epsilon \neq \emptyset \text{ for every } A\in p, \epsilon > 0\}.$$
Now given $p\in \Sa(G)$, we have $p\in i[\Sa(H)]$ iff $HU_\epsilon \in p$ for every $\epsilon > 0$. One direction is clear. For the other, if $HU_\epsilon \in p$ for every $\epsilon > 0$, it follows that for every $A\in p$ and $\epsilon > 0$, we have $AU_\epsilon \cap H\neq \emptyset$, and the collection 
$$\{B\subseteq_{op} H: B\cap AU_\epsilon \neq \emptyset \text{ for every } A\in p, \epsilon > 0\}$$
is a near ultrafilter $q$ on $H$ satisfying $i(q) = p$. 

From here on out, we will identify $\Sa(H)$ as a subspace of $\Sa(G)$ and suppress the embedding $i$. We now consider the quotient $\pi\colon G\to H\backslash G$ and extend it continuously to the respective Samuel compactifications. Given $p\in \Sa(G)$ and $q\in \Sa(H\backslash G)$, we have by Fact~\ref{Fact:SHGNults} that $\pi(p) = q$ iff $\pi^{-1}(AU_\epsilon)\in p$ for every $A\in q$ and $\epsilon > 0$. In particular, $\pi(p) = H$ iff $HU_\epsilon\in p$ for every $\epsilon > 0$. We obtain the following.
\vspace{2 mm}

\begin{prop}
	\label{Prop:PreimageSamuel}
	With $\pi\colon \Sa(G)\to \Sa(H\backslash G)$ the canonical map, we have $\pi^{-1}(\{H\}) = \Sa(H)$.
\end{prop}
\vspace{2 mm}

In the next section, we will be particularly interested in \emph{minimal} MHP flows. Recall that $S\subseteq G$ is called \emph{syndetic} if there is a finite set $F\subseteq G$ with $SF= G$. We have the following folklore fact.

\begin{fact}[\cite{Aus}, Ch.\ 1, Lem.\ 6]
	\label{Fact:SyndRet}
	Suppose $X$ is a $G$-flow and $x\in X$. Then $x\in X$ belongs to a minimal subflow iff for every $A\ni_{op} X$, the set $\{g\in G: xg\in A\}$ is syndetic.
\end{fact}
\vspace{2 mm}

The following simple proposition gives a combinatorial characterization for when $\Sa(H\backslash G)$ is minimal.
\vspace{2 mm}

\begin{prop}
	\label{Prop:SamuelMinimal}
	Let $H\subseteq G$ be a closed subgroup. Then the following are equivalent.
	\vspace{-2 mm}
	
	\begin{enumerate}
		\item 
		$\Sa(H\backslash G)$ is minimal.
		\item 
		For every $\epsilon > 0$, the set $HU_\epsilon \subseteq G$ is syndetic.
	\end{enumerate}
\end{prop}

\begin{rem}
	Compare this to the notion of co-precompactness, where $H\subseteq G$ is \emph{co-precompact} if $\Sa(H\backslash G) \cong \widehat{H\backslash G}$, the left completion of $H\backslash G$. This occurs iff for every $\epsilon > 0$, there is a finite $F\subseteq G$ with $HFU_\epsilon = G$.
\end{rem}

\begin{proof}
	First assume $\Sa(H\backslash G)$ is minimal. Since $H\in \Sa(H\backslash G)$ is a compatibility point, we have that $HU_\epsilon\subseteq H\b G$ is relatively open. Item $(2)$ then follows from minimality.
	
	Conversely, assume item $(2)$ holds. It follows that in $\Sa(H\backslash G)$, the return times of $H$ to any open neighborhood of $H$ are syndetic. Then by Fact~\ref{Fact:SyndRet}, $H\in \Sa(H\backslash G)$ belongs to a minimal subflow, and the orbit of $H$ is dense in $\Sa(H\backslash G)$.
\end{proof}
\vspace{2 mm}

Also in the next section, we will need to consider two closed subgroups $H, H'\subseteq G$ and understand when a $G$-map $\phi\colon \Sa(H\backslash G)\to \Sa(H'\b G)$ can exist.
\vspace{2 mm}

\begin{prop}
	\label{Prop:MapsSamuel}
	Suppose $H, H'\subseteq G$ are closed subgroups with both $\Sa(H\backslash G)$ and $\Sa(H'\b G)$ minimal. Then there is a $G$-map $\phi\colon \Sa(H\backslash G)\to \Sa(H'\b G)$ iff there is $g\in G$ with $H\subseteq g^{-1}H'g$.
\end{prop}

\begin{proof}
	For the forward direction, let $\phi$ be a $G$-map as above. By Proposition 14.1 in~\cite{AKL}, we know that $\phi$ must preserve the comeager orbit. In particular, $\phi(H) = H'g$ for some $g\in G$. It follows that for every $h\in H$, we have $H'gh = H'g$, i.e.\ that $H\subseteq g^{-1}H'g$. 
	
	For the reverse direction, if $H\subseteq g^{-1}Hg$ for some $g\in G$, it follows that $H$ stabilizes the point $H'g\in \Sa(H'\b G)$. Then the existence of a $G$-map $\phi$ as above follows from the universal property of $\Sa(H\backslash G)$.
\end{proof}
\vspace{2 mm}

\section{Canonical minimal flows}

In this section, we consider the universal minimal flow as well as two other "canonical" minimal flows in the context of Theorem~\ref{IntroThm:MHP}. These other special flows both deal with the notion of \emph{proximality}. 
\vspace{2 mm}

\begin{defin}
	\label{Def:Prox}
	Fix a $G$-flow $X$.
	\vspace{-2 mm}
	
	\begin{enumerate}
		\item 
		We say that $X$ is \emph{proximal} if for any $x, y\in X$, there is a net $g_i\in G$ and $z\in X$ with $xg_i\to z$ and $yg_i\to z$. Equivalently, there is $p\in \Sa(G)$ with $xp = yp$.
		\item 
		Let $P(X)$ denote the compact space of probability measures on $X$ endowed with the weak*-topology. Then $P(X)$ is also a $G$-flow. We say that $X$ is \emph{strongly proximal} if $P(X)$ is proximal. Equivalently, $X$ is strongly proximal iff $X$ is proximal and for any $\mu\in P(X)$, there is a net $g_i$ from $G$ with $\mu g_i\to \delta_x$ for some $x\in X$, where $\delta_x$ denotes the Dirac measure supported at $x$.
	\end{enumerate}
\end{defin}
\vspace{2 mm}

In \cite{G}, it is shown that there exist a \emph{universal minimal proximal flow}, denoted $\Pi(G)$, and a \emph{universal minimal strongly proximal flow}, denoted $\Pi_s(G)$ and often called the \emph{Furstenberg boundary}. Here, if P is a property of flows, a \emph{universal minimal P flow} is a minimal flow with property P which admits a $G$-map onto any other minimal flow with property P. Both are unique up to isomorphism. 
\vspace{2 mm}

\begin{lemma}
	\label{Lem:ProxRigid}
	If $X$ is a minimal, proximal $G$-flow, then the only $G$-map from $X$ to $X$ is the identity. 
\end{lemma}

\begin{proof}
	Suppose $\phi\colon X\to X$ is a $G$-map. If there is $x\in X$ with $\phi(x) = x$, then also $\phi(xp) = xp$ for every $p\in \Sa(G)$. As $X$ is minimal, this implies that $\phi$ is the identity map. Now suppose $\phi \neq \mathrm{id}_X$. Fix $x\in X$, and find $p\in \Sa(G)$ with $xp = \phi(x)p$. But as $\phi(x)p = \phi(xp)$, this is a contradiction since $\phi$ has no fixed points.  
\end{proof}
\vspace{2 mm}

\begin{lemma}
	\label{Lem:ProxMHP}
	The flows $\Pi(G)$ and $\Pi_s(G)$ are both MHP.
\end{lemma}

\begin{proof}
Suppose $\phi\colon X\to \Pi(G)$ is a non-trivial highly proximal $G$-map. Then it follows that $X$ is also minimal and proximal, so let $\psi\colon \Pi(G)\to X$ be a $G$-map. It follows that $\psi\circ \phi\colon X\to X$ is a non-trivial $G$-map, contradicting Lemma~\ref{Lem:ProxRigid}.

To show that $\Pi_s(G)$ is MHP, suppose $\phi\colon X\to \Pi_s(G)$ is a non-trivial highly proximal $G$-map. As a highly proximal extension of a minimal proximal flow, $X$ is proximal. Now suppose $\mu\in P(X)$. We can find a net $g_i\in G$ so that $\phi_*\mu g_i\to \delta_p$ for some $p\in \Pi_s(G)$. We may assume that $\mu g_i\to \nu$ for some $\nu\in P(X)$ supported on $\phi^{-1}(\{p\})$. Then since $\phi$ is highly proximal and $X$ is minimal, we can use Fact~\ref{Fact:HPOriginal} and find another net $h_j\in G$ so that $\phi^{-1}(\{p\})h_j$ shrinks down to some point $x\in X$. Hence $\nu h_j\to \delta_x$, showing that $X$ is strongly proximal. Now a similar argument to the proximal case shows that $\phi$ must be an isomorphism.
\end{proof}
\vspace{2 mm}

We can use $M(G)$ to create a particularly nice representation of $\Pi_s(G)$. Form the $G$-flow $P(M(G))$, and let $A\subseteq P(M(G))$ be a minimal \emph{affine} subflow of $P(M(G))$, i.e.\ a subflow which is closed under convex combinations and minimal with this property. Then $A$ is strongly proximal, and $\overline{ex(A)}$, the closure of the extreme points of $A$, is the unique minimal subflow of $A$. We then obtain $\overline{ex(A)}\cong \Pi_s(G)$. More details can be found in chapter 3 of \cite{G}.

From this characterization of $\Pi_s(G)$, it follows that a topological group $G$ is amenable iff $G$ admits no nontrivial minimal strongly proximal actions. As for proximal actions, we call $G$ \emph{strongly amenable} if $G$ admits no nontrivial minimal proximal actions. In particular, every strongly amenable group is amenable.
\vspace{2 mm}

\begin{lemma}
	\label{Lem:SubgroupProx}
	Let $X$ be a proximal $G$-flow, and let $H\subseteq G$ be a closed subgroup with $\Sa(H\backslash G)$ minimal. Then $H$ acts proximally on $X$.
\end{lemma}

\begin{proof}
	Let $x, y\in X$. As $X$ is a proximal $G$-flow, find $p\in \Sa(G)$ with $xp = yp$. Since $\Sa(H\backslash G)$ is minimal, we can find $q\in \Sa(G)$ with $pq\in \Sa(H)$. Then $xpq = ypq$, showing that $H$ acts proximally on $X$.
\end{proof}
\vspace{2 mm}

The following provides a generalization of Theorem 1.2 from \cite{MNT}.
\vspace{2 mm}

\begin{theorem}\mbox{}
	\label{Thm:ExAmSub}
	Fix a minimal MHP flow $X$ with a comeager orbit.
	\begin{enumerate}
		\item
		$X\cong M(G)$ iff $X\cong \Sa(H\backslash G)$ for some extremely amenable closed subgroup $H\subseteq G$.
		\item
		$X\cong \Pi_s(G)$ iff $X\cong \Sa(H\backslash G)$ for some maximal amenable subgroup $H\subseteq G$.
		\item 
		If $X\cong \Sa(H\backslash G)$ for some strongly amenable closed subgroup $H\subseteq G$ and $X$ is proximal, then $X\cong \Pi(G)$. 
		\end{enumerate}
\end{theorem}

\begin{proof}
	$(1)$ First assume $X\cong M(G)$, and let $H\subseteq G$ be the closed subgroup given by item $(4)$ of Theorem~\ref{IntroThm:MHP}. Fix a minimal subflow $M\subseteq \Sa(G)$, and consider the canonical map $\pi\colon \Sa(G)\to \Sa(H\backslash G)$. Then $\pi$ is surjective, and $\pi|_M$ is an isomorphism. Since by Proposition~\ref{Prop:PreimageSamuel} we have $\pi^{-1}(\{H\}) = \Sa(H)$, it follows that $M\cap \Sa(H)$ is a singleton and an $H$-flow. As any minimal subflow of $\Sa(H)$ is isomorphic to $M(H)$, we see that $H$ is extremely amenable. 
	
	Conversely, suppose $H\subseteq G$ is an extremely amenable closed subgroup of $G$ with $\Sa(H\backslash G)$ minimal. Then $M(G)$ must have an $H$-fixed point. It follows that there is a $G$-map $\phi\colon \Sa(H\backslash G)\to M(G)$. As we assumed that $\Sa(H\backslash G)$ was minimal, it follows that $\phi$ is an isomorphism.  
	\vspace{2 mm}
	
	$(2)$ We break the argument into the following parts.
	\begin{itemize}
		\item 
		If $X \cong \Pi_s(G)$, then $X\cong \Sa(H\backslash G)$ with $H\subseteq G$ a closed amenable subgroup.
		\item 
		If $X\cong \Sa(H'\b G)$ with $H'\subseteq G$ a closed amenable subgroup, then $X$ maps onto any strongly proximal flow.
	\end{itemize}
	From these two items, the theorem follows, since if $H\subsetneq H'$ are both closed amenable subgroups of $G$, then by Proposition~\ref{Prop:MapsSamuel}, we have a non-trivial factor map $\Sa(H\backslash G)\to \Sa(H'\b G)$. If we had $\Pi_s(G)\cong \Sa(H\backslash G)$, then the second item would allow us to build a non-trivial $G$-map from $\Pi_s(G)$ to itself, contradicting Lemma~\ref{Lem:ProxRigid}. Conversely, if $X\cong \Sa(H'\b G)$ for $H'\subseteq G$ a maximal amenable subgroup, then using the second item we obtain a map $\Sa(H'\b G)\to \Pi_s(G)$. By the first item, we have $\Pi_s(G)\cong \Sa(H\backslash G)$ for some closed amenable subgroup $H\subseteq G$. By Proposition~\ref{Prop:MapsSamuel} we must have $H\subseteq g^{-1}H'g$, so in fact $H = g^{-1}H'g$ as $H$ was assumed maximal. It follows that $\Sa(H\backslash G)\cong \Sa(H'\b G) = \Pi_s(G)$.
	
	To prove the first item, suppose $X\cong \Pi_s(G) \cong \Sa(H\backslash G)$. Let $M\subseteq \Sa(G)$ be a minimal subflow, and let $A\subseteq P(M)$ be a minimal affine subflow. Then $X\cong \overline{ex(A)}$, the unique minimal subflow of $A$. Now letting $\pi\colon \Sa(G)\to \Sa(H\backslash G)$ be the canonical map, we have the affine extension $\pi_*\colon P(\Sa(G))\to P(\Sa(H\backslash G))$ to the spaces of measures. Identifying each $p\in \Sa(H\backslash G)$ with the Dirac measure $\delta_p$, we have that $\Sa(H\backslash G)$ is the unique minimal subflow of $P(\Sa(H\backslash G))$. It follows that $\pi_*|_{\overline{ex(A)}}\colon \overline{ex(A)}\to \Sa(H\backslash G)$ is an isomorphism. However, we also have $\pi_*^{-1}(\{H\}) = P(\Sa(H))$, so $P(\Sa(H))\cap \overline{ex(A)}$ is a singleton and an $H$-flow, i.e.\ an $H$-invariant measure on $\Sa(H)$. Hence $H$ is amenable.
	
	To prove the second item, we assume $X\cong \Sa(H'\b G)$ with $H'\subseteq G$ a closed amenable subgroup. On $P(\Pi_s(G))$, $H'$ acts proximally by Lemma~\ref{Lem:SubgroupProx}, hence $H'$ acts strongly proximally on $\Pi_s(G)$. Since $H'$ is amenable, it follows that $\Pi_s(G)$ has an $H'$-fixed point, so there is a $G$-map from $\Sa(H'\b G)$ to $\Pi_s(G)$.
	\vspace{2 mm}
	
	$(3)$ As for the third item, we assume that $X\cong \Sa(H\backslash G)$ is proximal and that $H\subseteq G$ is strongly amenable. By Lemma~\ref{Lem:SubgroupProx}, $H$ acts proximally on $\Pi(G)$. As $H$ is strongly amenable, $\Pi(G)$ has an $H$-fixed point, so there is a $G$-map from $\Sa(H\backslash G)$ to $\Pi(G)$. As $\Sa(H\backslash G)$ was assumed proximal, we have $\Sa(H\backslash G)\cong \Pi(G)$.
\end{proof}
\vspace{2 mm}

\begin{rem}
	When considering the Furstenberg boundary or the universal minimal proximal flow of locally compact groups, we note that if  $\Sa(H\backslash G)$ is minimal, then in fact $\Sa(H\backslash G) = H\b G$, i.e.\ that $H$ is a cocompact subgroup of $G$. This is because $H\b G\subseteq \Sa(H\backslash G)$ is comeager, but also $F_\sigma$, being an orbit of a locally compact group action. So $\Sa(H\b G)\setminus (H\b G)$ is $G_\delta$, and if it were non-empty, then by minimality it would be dense, a contradiction. 
\end{rem}
\vspace{2 mm}

The following question addresses whether item $(3)$ in Theorem~\ref{Thm:ExAmSub} can be strengthened to have the same form as items $(1)$ and $(2)$.
\vspace{2 mm}

\begin{que}
	\label{Que:ProxSubgroup}
	Suppose $\Pi(G)\cong \Sa(H\backslash G)$ for some closed subgroup $H\subseteq G$. Then must $H$ be strongly amenable?
\end{que}
\vspace{2 mm}

\section{Reflecting meager orbits}

The main theorem of this section is the following ``reflection'' theorem. 

\begin{theorem}
	\label{Thm:Reflection}
	Let $X$ be a minimal MHP flow all of whose orbits are meager. Then there is a factor $\phi\colon X\to Y$ so that $Y$ is metrizable and also has all orbits meager.
\end{theorem}

Therefore in addition to the notation of the previous sections, we assume that $X$ is minimal and does not have a comeager orbit. 

The metrizable factor of $X$ that we produce will be a space of uniformly continuous functions from $G$ (with its left-invariant metric uniformity) to a compact metric space. If $Y$ is a compact metric space, then $Y^G$ is a compact space when endowed with the product topology. The group $G$ acts on $Y^G$ by shift, where for $y\in Y^G$ and $g, h\in G$, we have $y\cdot g(h) = y(gh)$. Now suppose $y\in Y^G$ is uniformly continuous. Then $\overline{y\cdot G}$ is a uniformly equi-continuous family, and furthermore, the space $\overline{y\cdot G}$ is metrizable. To see why the last claim is true, note that pointwise convergence of a net of uniformly equi-continuous functions is determined by pointwise convergence on some countable dense subset of $G$. 

In order to obtain factors of $X$, we use functions which arise from $X$ in the following way. Suppose $f\colon X\to Y$ is continuous, and fix $x\in X$. Then we obtain a uniformly continuous function $f_x\colon G\to Y$ via $f_x(g) = f(xg)$. Then notice that $f_x\cdot g = f_{xg}$, and if $x_i\to y$, then $f_{x_i}\to f_y$. It follows that the map $x\to f_x$ is a surjective $G$-map of $X$ onto $\overline{f_x\cdot G}$.

We now turn towards the proof of the theorem. Our first task is to provide a ``global'' version of item (2) from Lemma~\ref{Lem:Compatibility}. This doesn't require minimality.
\vspace{2 mm}

\begin{lemma}
	\label{Lem:GlobalNonComp}
	Suppose $Z$ is an MHP flow with no comeager orbit. Then there is some $c > 0$ and $A\subseteq_{op} Z$ with $x[c]$ nowhere dense for every $x\in A$. 
\end{lemma}

\begin{proof}
	Notice by the lower semi-continuity of $\partial$ that $x[c]$ is closed for every $c > 0$.
	Suppose towards a contradiction that for every $c > 0$, the set $D_c := \{x\in Z: \mathrm{int}(x[c]) \neq \emptyset\}$ is dense. Using adequacy, we see that for every $c > 0$, we have $D_{c/3}\subseteq E_c := \{x\in Z: x\in \mathrm{int}(x[c])\}$, so $E_c$ is also dense. Then $\partial$ is compatible at any point in the comeager set $\bigcap_{c > 0}\bigcup_{x\in E_c} \mathrm{int}(x[c])$. Theorem~\ref{IntroThm:MHP} then shows that $Z$ has a comeager orbit, contradicting our assumption. 
\end{proof}
\vspace{3 mm}

Fix $c > 0$ and $A\subseteq_{op} X$ as given by Lemma~\ref{Lem:GlobalNonComp}. Fix $D\subseteq X$ a countable dense set, and write $[D]^2 = \{\{p_i, q_i\}: i < \omega\}$. Keeping in mind Corollary~\ref{Cor:Topometric}, find $\gamma_i\in C_{OL}(X, [0,1])$ with $|\gamma_i(p_i) - \gamma_i(q_i)| > \partial(p_i, q_i)/2$. Let $\gamma\colon X\to [0,1]^\omega$ be the concatenation of the $\gamma_i$. It will be helpful to view $[0,1]^\omega$ as a topometric space whose metric is given by the uniform distance $d_u$.

\begin{lemma}
	\label{Lem:UniformBallPreimage}
	Let $B\subseteq [0,1]^\omega$ be a closed $d_u$-ball of radius $c/4$. Then $\gamma^{-1}(B)\cap A \subseteq X$ is nowhere dense.
\end{lemma} 

\begin{proof}
	As $\gamma^{-1}(B)\cap A$ is relatively closed in $A$, we show that it has empty interior. Let $W\subseteq A$ be non-empty open. Pick $p\in W\cap D$. Then $p[c]$ is a closed, nowhere dense set, so find $q\in (W\setminus p[c])\cap D$. Then $p, q\in W$ with $\partial(p, q) > c$. Suppose that $\{p, q\} = \{p_k, q_k\}$. Then $|\gamma_k(p) - \gamma_k(q)| > c/2$. In particular, $d_u(\gamma(p), \gamma(q)) > c/2$, so $W\not\subseteq \gamma^{-1}(B)$.
\end{proof}
\vspace{2 mm}

\begin{lemma}
	\label{Lem:SyndImage}
	Suppose $Z$ is a minimal $G$-flow, $x\in Z$, and $S\subseteq G$ is syndetic. Then $x\cdot S\subseteq Z$ is somewhere dense.
\end{lemma}

\begin{proof}
	Since $S\subseteq G$ is syndetic, find $g_0,\ldots,g_{k-1}\in G$ with $\bigcup_{i< k} Sg_i = G$. Then $\bigcup_{i< k}(x\cdot S)\cdot g_i = x\cdot G\subseteq Z$ is dense, so $x\cdot Sg_i$ is somewhere dense for some $i< k$. Then by translating, $x\cdot S$ is somewhere dense as well.
\end{proof}
\vspace{2 mm}

Now let $\alpha\colon X\to [0,1]$ be a continuous function with $\alpha^{-1}(\{1\})\neq \emptyset$ and \\ $\alpha[X \setminus A] = \{0\}$. Form the function $\theta = \alpha\times \gamma\colon X\to [0,1]\times [0,1]^\omega$. Pick $p\in X$, and then form $\theta_p\colon G\to [0,1]\times [0,1]^\omega$.

We will show that $\overline{\theta_p\cdot G}$ has all orbits meager.	Towards a contradiction, suppose $\theta_q\in \overline{\theta_p\cdot G}$ belonged to a comeager orbit; as $\{q\in X: \alpha(q) > 3/4\}\subseteq X$ is open, we may assume that $q$ belongs to this set. Let $r >0$ be small enough so that both $r < c/4$ and for $g\in U_r$, we have $\alpha(qg) > 1/2$. By the Effros theorem, $\theta_q \cdot U$ is a relatively open subset of $\theta_q\cdot G$. By Fact~\ref{Fact:SyndRet}, it follows that $S := \{g\in G: \theta_q\cdot g\in \theta_q\cdot U\}$ is syndetic, so by Lemma~\ref{Lem:SyndImage}, $q\cdot S\subseteq X$ is somewhere dense. Furthermore, for $g\in G$, we have
\begin{align*}
\theta_q\cdot g(1_G) &=  \theta(qg)\\[1 mm]
&= (\alpha(qg), \gamma(qg)).
\end{align*}
It follows that for $g\in S$, there is $h\in U$ with $\alpha(qg) = \alpha(qh) > 1/2$. Hence $q\cdot S\subseteq A$. However, by item (1) of Proposition~\ref{Prop:TopometricDynamics}, $\gamma[q\cdot S] = \gamma[q\cdot U]$ lies in a $d_u$-ball of radius $c/4$, contradicting Lemma \ref{Lem:UniformBallPreimage}.
\vspace{2 mm}

\begin{que}
	\label{Que:NeedHP}
	Theorem~\ref{Thm:Reflection} shows that for \emph{any} non-metrizable minimal $G$-flow $X$ all of whose orbits are meager, we have a factor $\phi\colon S_G(X)\to Y$ where $Y$ is metrizable and has all orbits meager. Is it necessary to pass to the universal highly proximal extension? More precisely, is there an example of a Polish group $G$ and a minimal $G$-flow $X$ with all meager orbits, but all of whose metrizable factors have a comeager orbit? 
\end{que}

\section{Distal universal minimal flows}

As an application of Theorem~\ref{Thm:Reflection}, we prove Theorem~\ref{Thm:DistalFlows}, a characterization of when a Polish group $G$ has distal universal minimal flow.
\vspace{2 mm}

\begin{defin}\mbox{}
	\label{Def:Distal}
	A $G$-flow $X$ is called \emph{distal} if for any pair of points $x\neq y\in X$ and any net $g_i$ from $G$ with $xg_i\to z\in X$, we have $yg_i\not\to z$.
\end{defin}
\vspace{0 mm}

\begin{theorem}
	\label{Thm:DistalFlows}
	Let $G$ be a Polish group, and assume that $M(G)$ is distal. Then $M(G)$ is metrizable.
\end{theorem}
\vspace{2 mm}

In \cite{MNT}, the authors consider Polish groups which are \emph{strongly amenable}, groups which admit no non-trivial minimal proximal flows. They prove that if $G$ is strongly amenable and $M(G)$ is metrizable, then $G$ has a closed, normal, extremely amenable subgroup $H$ with $G/H$ compact and $M(G) \cong G/H$. As any group $G$ with $M(G)$ distal is also strongly amenable, we obtain the following corollary.
\vspace{2 mm}

\begin{cor}
	\label{Cor:DistalForm}
	Let $G$ be a Polish group with $M(G)$ distal. Then $G$ has a closed, normal, extremely amenable subgroup $H$ with $G/H$ compact and $M(G)\cong G/H$.
\end{cor}
\vspace{2 mm}

We briefly review some facts about enveloping semigroups and distal flows; see \cite{Aus} for more detail. To any $G$-flow $X$, we can associate to it the \emph{enveloping semigroup} $E(X)$. Given $g\in G$, form the function $\rho_g\colon X\to X$ given by $\rho_g(x) = xg$. Then $E(X)$ is the closure of the set $\{\rho_g: g\in G\}$ in the compact space $X^X$. Each $f\in E(X)$ is a function, and because we take our $G$-flows to be right actions, it will be more convenient to write function application and composition on the right, i.e.\ for $x\in X$ and $f\in E(X)$, we write $xf$ instead of $f(x)$. Then $E(X)$ becomes a compact left-topological semigroup, in particular a $G$-flow, where $f\cdot g = \rho_g\circ f$. For any $x\in X$, the map $\lambda_x\colon E(X)\to X$ given by $\lambda_x(f) = xf$ is a $G$-map. 

When $X$ is distal, then $E(X)$ is a group. Furthermore, if $X$ is also minimal, then $E(X)$ is a minimal distal system. If $f\in E(X)$, then the left multiplication map $\lambda_f$ is a $G$-flow automorphism. In particular, if $M(G)$ is distal, then $E(M(G))\cong M(G)$, and for any $p, q\in M(G)$, there is a $G$-flow automorphism $\phi$ with $\phi(p) = q$. 

In the proof of Theorem~\ref{Thm:DistalFlows}, we will need the following simple proposition.
\vspace{2 mm}

\begin{prop}[\cite{Aus}, Cor.\ 7(c)]
	\label{Prop:NonDistalLifts}
	Let $Y$ be a distal flow, and let $\phi\colon Y\to X$ be a factor. Then $X$ is also distal
\end{prop}
\vspace{2 mm}

We will also need to recall the main result of \cite{ZucProx}.
\vspace{2 mm}

\begin{fact}[\cite{ZucProx}, Cor.\ 3.3]
	\label{Fact:HPNonMetr}
	If $X$ is a minimal, metrizable flow with all orbits meager, then the universal highly proximal extention $S_G(X)$ is non-metrizable. In particular, the map $\pi_X\colon S_G(X)\to X$ is a non-trivial highly proximal extension.
\end{fact}
\vspace{2 mm}

We can now complete the proof of Theorem~\ref{Thm:DistalFlows}. Towards a contradiction, suppose $M(G)$ were distal, but not metrizable. Then by Theorem~\ref{Thm:Metrizability}, we have $|M(G)| = 2^\mathfrak{c}$, so in particular $M(G)$ contains more than one orbit. As there is a $G$-flow automorphism bringing any one orbit to any other, we see that $M(G)$ contains all meager orbits. By Theorem~\ref{Thm:Reflection}, let $X$ be a minimal metrizable flow with all meager orbits. Then by Fact~\ref{Fact:HPNonMetr}, $\pi_X\colon S_G(X)\to X$ is a non-trivial highly proximal extension of minimal flows, which implies that $S_G(X)$ is not distal. Now let $\phi\colon M(G)\to S_G(X)$ be a $G$-map. By Proposition~\ref{Prop:NonDistalLifts}, we must also have $M(G)$ not distal, completing our contradiction.

\noindent
Andy Zucker

\noindent
Universit\'e Paris Diderot

\noindent
andrew.zucker@imj-prg.fr

\end{document}